\documentclass[11pt]{amsart}
\usepackage{amssymb,amsmath,txfonts,mathrsfs,mathtools,titletoc, tikz, stmaryrd}
\usepackage[alphabetic]{amsrefs}
\usepackage{float}
\usepackage[colorlinks, linkcolor=blue,anchorcolor=Periwinkle,
citecolor=red,urlcolor=Emerald]{hyperref}
\usepackage{xcolor}
\usepackage{enumitem}
\usepackage{geometry,array} 
\usepackage{graphicx}
\usepackage{subfigure}
\usepackage{bookmark}
\usepackage{tikz}\usetikzlibrary{matrix}
\usepackage{url}
\usepackage{braids}

\usetikzlibrary{decorations.markings}
\tikzset{->-/.style={decoration={  markings,  mark=at position #1 with
			{\arrow{>}}},postaction={decorate}}}
\tikzset{-<-/.style={decoration={  markings,  mark=at position #1 with
			{\arrow{<}}},postaction={decorate}}}
\usepackage[all]{xy}
\usetikzlibrary{arrows}
\usetikzlibrary{decorations.pathreplacing,decorations.pathmorphing,shadings,fadings,calc}
\usepackage{extarrows}
\usepackage{tikz-3dplot}

\usetikzlibrary{matrix}
\usepackage[all]{xy}
\usetikzlibrary{arrows,shapes}
\usetikzlibrary{calc}
\tikzset{->-/.style={decoration={  markings,  mark=at position #1 with
			{\arrow{>}}},postaction={decorate}}}
\tikzset{-<-/.style={decoration={  markings,  mark=at position #1 with
			{\arrow{<}}},postaction={decorate}}}
\usepackage[]{hyperref}
\numberwithin{equation}{section}

\usetikzlibrary{decorations.markings}
\tikzset{->-/.style={decoration={  markings,  mark=at position #1 with
			{\arrow{>}}},postaction={decorate}}}
\tikzset{-<-/.style={decoration={  markings,  mark=at position #1 with
			{\arrow{<}}},postaction={decorate}}}

\usepackage{lscape}
\usepackage{CJK}

\newtheorem{theorem}{Theorem}[section]

\newtheorem{lemma}[theorem]{Lemma}
\newtheorem{remark}[theorem]{Remark}

\newtheorem{cor}[theorem]{Corollary}

\numberwithin{equation}{section}

\def\R{{\bf\mathbb R} }

\begin{document}
\title[Brezis-Merle inequality for $k$-Hessian equation]{Uniform estimates and Brezis-Merle type inequalities for  the $k$-Hessian equation}
\email{}
\date{\today}

\author{Jie Deng}
\author{Haibin Wang}
\author{Bin Zhou}

\address{School of Mathematical Sciences, Peking University, Beijing 100871, China}
\email{dj0401@stu.pku.edu.cn}
\address{School of Mathematical Sciences, Peking University, Beijing 100871, China}
\email{haibinwang@pku.edu.cn}
\address{School of Mathematical Sciences, Peking University, Beijing 100871, China}
\email{bzhou@pku.edu.cn}

\thanks {This research is partially supported by  National Key R$\&$D Program of China 2023YFA009900 and NSFC  Grant 12271008.}

\begin{abstract}
In this paper, we prove a Brezis-Merle type inequality for $k$-convex functions vanishing on the boundary. 
As an application, we establish an Alexandrov-Bakelman-Pucci type estimate for the intermediate Hessian equation. 
Furthermore, we establish a concentration-compactness principle for the blow-up behavior of solutions to
the mean field type $k$-Hessian equation.
\end{abstract}
	


\maketitle
	
\titlecontents{section}[0em]{}{\hspace{.5em}}{}{\titlerule*[1pc]{.}\contentspage}
\titlecontents{subsection}[1.5em]{}{\hspace{.5em}}{}{\titlerule*[1pc]{.}\contentspage}
\tableofcontents
	
\section{Introduction}

Let $\Omega$ be a bounded smooth domain in $\R^n$. For $u\in C^2(\Omega)$, $k=1,\dots, n$, the {\it $k$-Hessian} operator $S_k$ is 
\begin{align}
S_k[u]=\sum_{1\le i_1<...<i_k\le n}\lambda_{i_1}\cdots\lambda_{i_k}, \nonumber
\end{align}
where $\lambda_1,...,\lambda_n$ are the eigenvalues of the Hessian matrix $D^2u$. When $k=1$, it is the Laplacian operator.
Let 
\[\Phi^k(\Omega)=\{u\in C^2(\Omega)\, | \, S_j[u]\ge 0 ,  \forall j=1,...,k\}\]
be the set of smooth $k$-convex functions. $k$-convex functions are admissible functions
for $S_k$ to be elliptic.
 The notion of $k$-convexity can be extended to general upper semicontinuous functions.
An upper semicontinuous function $u:\Omega \to [-\infty, +\infty)$ is called {\it $k$-convex} in $\Omega$ if $S_k[q]\geq 0$ for all quadratic polynomials $q$ for which the diﬀerence $u-q$ has a finite local maximum in $\Omega$.
In a series of papers \cites{TW97, TW99, TW02}, Trudinger-Wang introduced the Hessian measures for general $k$-convex functions and
developed the weak solution theory for the $k$-Hessian equation
\begin{equation}\label{k-hess-eq}
S_k[u]=f.
\end{equation}
A key property shown in  \cite{TW99} is the weak continuity of $S_k[u]$. Namely, if a sequence of $k$-convex functions $\{u_m\}$ converges  to a $k$-convex function $u$ in $L^1_{\text{loc}}(\Omega)$, then the $k$-Hessian measures $\mu_k[u_m]$ converge weakly to the $k$-Hessian measure $\mu_k[u]$.

The uniform estimate for the  $k$-Hessian equations was established by \cite{CW01} under the assumption that 
$f\in L^p(\Omega)$ where $p>\frac{n}{2k}$ if $k\leq \frac{n}{2}$ or $p=1$ if $k>\frac{n}{2}$,
 based on the Hessian Sobolev inequalities \cites{Wan94, Wan09}.

In a famous paper \cite{BM91}, Brezis and Merle investigated uniform estimates and blow-up behavior 
for equation
\[-\triangle u=V(x) e^u\]
based on the following basic inequality:
\[\int_\Omega \exp\left[\frac{(4\pi-\delta) |u(x)| }{\|\triangle u\|_{L^1(\Omega)}}\right]\,dx\leq \frac{4\pi^2}{\delta}(\mathrm{diam}(\Omega))^2\]
for all $\delta\in (0,4\pi)$ and $u$ on a bounded domain $\Omega\subset\mathbb R^2$ with vanishing boundary value. Subsequently, in higher dimensions, Aguilar and Peral \cite{AP94} established the inequality
\[\int_\Omega \exp\left[\frac{(n^{\frac{n}{n-1}}\omega_{n}^{\frac{1}{n-1}}-\delta) |u(x)| }{\|f\|^{\frac{1}{n-1}}_{L^1(\Omega)}}\right]\,dx\leq \frac{n^{\frac{n}{n-1}}\omega_{n}^{\frac{1}{n-1}}}{\delta}(\mathrm{diam}(\Omega))^2\]
for $\Omega \subset \mathbb{R}^n$, $\delta \in (0, n^{\frac{n}{n-1}}\omega_{n}^{\frac{1}{n-1}})$, where $\omega_n$ denotes the volume of the unit ball in $\mathbb{R}^n$, and $u$ vanishes on the boundary with $-\operatorname{div}(|\nabla u|^{n-2}\nabla u) = f \in L^1(\Omega)$.
An inequality of similar form for the complex Monge-Amp\`ere operator was obtained by \cites{De09, AC09, BB}.
In this paper, we study the analogous theory of $k$-Hessian equations.

Let $\Phi^k_0(\Omega)$ be the subset of functions in $\Phi^k(\Omega)$ with vanishing boundary value.
When $\partial \Omega$ is $(k-1)$-convex, i.e., there exists a constant $c_0>0$ such that
\[
\sum_{1\le i_1<\dots<i_j\le n-1}\kappa_{i_1}(x)\cdots \kappa_{i_j}(x)\ge c_0
\]
for all $x\in\partial\Omega$ and $1\le j\le k-1$, where $\kappa_1(x),\dots,\kappa_{n-1}(x)$ are the principal curvatures of $\partial\Omega$ at $x$, we have $\Phi^k_0(\Omega)\neq\emptyset$.
We define the {\it $k$-Hessian integral}
\[I_k(u)=\int_{\Omega}(-u)S_k[u]\,dx\]
and the {\it $k$-Hessian mass}
\[\mathcal M_k(u)=\int_{\Omega}S_k[u]\,dx.\]
Then 
\[||u||_{\Phi^k_0(\Omega)}:=I_k(u)^{\frac{1}{k+1}}\] 
is a norm on $\Phi^k_0(\Omega)$.  As we mentioned above, the $k$-Hessian measure can be defined for 
general $k$-convex functions. We denote by $\mathcal F^k(\Omega)$ the set of all $k$-convex functions $u$
such that $\mathcal M_k(u)=\mu_k[u](\Omega)<+\infty$ and $\displaystyle\lim_{y\to x}u(y)=0$ for any $x\in\partial\Omega$.

The first result of this paper is the following Brezis-Merle type inequality for the $k$-Hessian equations.

\begin{theorem}\label{thm BM inequality} Let $u\in \mathcal F^k(\Omega)$.
\begin{enumerate}
\item [(1)] If  $1\leq k<n/2$, then for $1\leq p \leq \frac{kn}{n-2k}$, there exists $C=C(n, k, p, |\Omega|)$ such that
\[ ||u ||_{L^p(\Omega)}\leq  C\mathcal M_k(u)^{1/k}.\]

\item [(2)] If $k = \frac{n}{2}$, then for $0 < \lambda < \alpha_0$ and $1 \leq \beta \leq \beta_0$, there exists $C=C(n)$ such that the following Brezis--Merle type inequality holds:
\begin{align}
\sup\left\{\int_{\Omega}\exp\left(\lambda\left(\frac{-u}{\mathcal{M}_{k}(u)^{\frac{1}{k}}}\right)^{\frac{k\beta}{k+1}}\right)\,dx: u\in \mathcal F^k(\Omega),\,0<||u||_{\Phi_0^k(\Omega)}<\infty\right\}\leq \frac{\alpha_0C(n) }{\alpha_0-\lambda}|\Omega|,
\end{align}
\end{enumerate}
where
\begin{equation}\label{apbe}
\alpha_0 = n \left[ \omega_{n} \binom{n}{k} \right]^{2/n}, \quad 
\beta_0 = \frac{n+2}{n}.
\end{equation}
\end{theorem}

With the Brezis-Merle type inequality in the intermediate case($k=\frac{n}{2}$), we can improve the uniform estimate and get an ABP type estimate.

 \begin{theorem}\label{thm L infty estimate}
Let $\Omega$ be a bounded $(k-1)$-convex smooth domain, $k=n/2$ and $u \in C^2(\Omega) \cap C^0(\overline{\Omega})$. Let $\Phi: \mathbb{R} \to \mathbb{R}^+$ be a positive increasing function with 
\[\int_0^\infty \Phi^{-1/k}(t) \,dt \le \Lambda < \infty.\]  
Then there exist constants $c_1=c_1(n,|\Omega|), c_2=c_2(n,\Lambda)>0$ such that
\begin{align}\label{k hessian L infty estimate}
\sup_{\Omega} u \le \sup_{\partial \Omega}u+c_1 + c_2 \left( \int_{\Gamma^-} S_k[-u] \Phi(\log(S_k[-u])) \,dx \right)^{1/k},
\end{align}
where $\Gamma^- = \{ x \in \Omega \mid u(x) > 0 \text{ and } S_k[-u](x) > 0 \}$.
\end{theorem}
 
We can also apply the Brezis-Merle type inequality to study the 
concentration‑compactness principle for the intermediate Hessian equation. Assume $k=n/2$.
For $\gamma>0$ and $u\in\Phi_0^k(\Omega)$, we introduce the 
functional
\begin{equation}\label{g-func}
	G_{\gamma}(u)=\frac{1}{\gamma}\log\bigg(\frac{1}{|\Omega|}\int_{\Omega}
	e^{-\gamma u}\,dx\bigg)+\frac{1}{k+1}\int_{\Omega}u\,S_k[u]\,dx,
\end{equation}
whose Euler--Lagrange equation is the mean field type $k$-Hessian equation
\begin{equation}\label{mean field equation of k}
	S_k[u]=\frac{e^{-\gamma u}}{\int_{\Omega}e^{-\gamma u}\,dx}.
\end{equation}
Note that any solution of \eqref{mean field equation of k} automatically satisfies 
$\int_{\Omega}S_k[u]\,dx=1$.  
When $k=\frac{n}{2}$, the critical parameter for the Moser--Trudinger type 
inequality is $\frac{k+1}{k}\alpha_0$, where $\alpha_0$ is given in \eqref{apbe}.  

Given a sequence $\{\gamma_j\}$ increasing to $\frac{k+1}{k}\alpha_0$, let 
$u_j\in\Phi_0^k(\Omega)$ be corresponding solutions of 
\eqref{mean field equation of k} that maximize $G_{\gamma_j}$.  
A point $x_0\in\Omega$ is called a \emph{blow-up point} for the sequence 
$\{u_j\}$ if there exists a sequence $\{x_j\}\subset\Omega$ converging to 
$x_0$ such that $\lim\limits_{j\to\infty}u_j(x_j)=-\infty$.  
With these notions we establish the following concentration‑compactness 
principle.

\begin{theorem}\label{mean-field-compactness}
	Assume that $u_j\to u$ in $L^1_{\rm loc}(\Omega)$ and that $\{u_j\}$ has 
	no blow-up points on the boundary $\partial\Omega$.  Then one of the 
	following alternatives holds:
	\begin{enumerate}
		\item[(1)] The limit $u$ is a weak solution of 
		\eqref{mean field equation of k} with $\gamma=\frac{k+1}{k}\alpha_0$.
		\item[(2)] There exists $x_0\in\Omega$ such that $\mu_k[u]=\delta_{x_0}$, and
		\[
		u(x)=\Bigl[\binom{n}{k}\omega_n\Bigr]^{\frac{1}{k}}\log|x-x_0|+O(1).
		\]
	\end{enumerate}
\end{theorem}

The rest of the paper is organized as follows.  In Section~\ref{BM-pf} we 
prove the Brezis-Merle type inequality (Theorem~\ref{thm BM inequality}).  
Section~\ref{ABP-pf} is devoted to an Alexandrov--Bakelman--Pucci type 
estimate for the $k$-Hessian equation.  In Section~\ref{mean-field eq} we 
introduce the associated functionals for the mean field type equation, and 
establish their basic properties.  Finally, in Section~\ref{con-com} we 
apply the Brezis-Merle inequality and the 
Moser--Trudinger inequality to prove the concentration‑compactness 
principle (Theorem~\ref{mean-field-compactness}).
\vskip 20pt

\section{The Brezis-Merle type inequality}\label{BM-pf}

We make use of the isocapacitary inequalities for the $k$-Hessian equations to establish the Brezis-Merle type inequalities.

Recall the capacity for the $k$-Hessian equation \cites{TW02, Lab02}
\[
\operatorname{Cap}_{k}(K,\Omega)=\sup\left\{\int_{K}S_k[u]\,dx: u \; \text{is $k$-convex}, -1<u<0\right\}.
\]
A non-negative Borel measure $\mu$ on $\Omega$ is said to be \emph{continuous with respect to the $k$-Hessian capacity} if for every $\varepsilon > 0$, there exists $\delta > 0$ such that for any open set $E \subset \Omega$ with $\operatorname{Cap}_k(E, \Omega) < \delta$, we have $\mu(E) < \varepsilon$.
It follows from \cite{TW02} that $k$-Hessian measure $\mu_k[u]$ is continuous with respect to capacity if either $u$ is bounded in $\Omega$ or there exist constants $\varepsilon>0$ and $C>0$ such that 
\[\mu_k[u](B_r(x)) \le C r^{n-2k+\varepsilon}\]
whenever $B_r(x)\subset\Omega$. We will use the following isocapacitary inequalities.

\begin{theorem}\cite[Theorem 1.2]{XZ14}\label{iso}
	\begin{enumerate}
		\item[(i)] If $1\leq k<\frac{n}{2}$ and $1\leq q \leq \frac{n(k+1)}{n-2k}$, there exists a constant $C=C(n,k,q,|\Omega|)>0$ such that for any $E\subset\Omega$,
		\[
		|E| \leq C\cdot \operatorname{Cap}_k(E,\Omega)^{\frac{q}{k+1}}.
		\]
		
		\item[(ii)] If $k = \frac{n}{2}$ and $1\leq \beta \leq 1+\frac{1}{k}$, there exists a constant $C(n)>0$ such that for any $E\subset\Omega$,
		\[
		|E|\cdot \exp\left(\alpha_0\frac{1}{\operatorname{Cap}_{k}(E,\Omega)^{\frac{\beta}{k+1}}}\right)\leq C(n)|\Omega|,
		\]
		where $\alpha_0$ is given in \eqref{apbe}.
	\end{enumerate}
\end{theorem}

In order to establish the Brezis-Merle inequality for general $k$-convex functions, we need the following comparison principle.
\begin{theorem}\cite[Theorem 3.3]{TW02}\label{comparison}
	Let $\Omega \subset \mathbb{R}^n$ be a bounded domain and let $u, v$ be $k$-convex functions such that $u = v = \phi$ continuously on $\partial \Omega$. If the $k$-Hessian measures $\mu_k[u]$ and $\mu_k[v]$ are both continuous with respect to the $k$-Hessian capacity, then
	\[
	\mu_k[u]\bigl(\{ u < v \}\bigr) \ge \mu_k[v]\bigl(\{ u < v \}\bigr).
	\]
\end{theorem}
\begin{remark}
	The theorem also holds if the condition ``$u = v = \phi$ continuously on $\partial\Omega$'' is replaced by $u \ge v$ on $\partial\Omega$.
\end{remark}

\begin{proof} [Proof of Theorem \ref{thm BM inequality}] 
	Define $K_{t}:=\{x\in \Omega : u(x)<-t\}$, $t>0$.  Let $u_{K_t}$ be the extremal function with respect to $\operatorname{Cap}_{k}(K_{t},\Omega)$ and let $u_j \in \Phi_0^k(\Omega)$ be a sequence converging to $u$ decreasingly. Choose $\delta>0$ such that there exists a compact set $B_{\delta}$ satisfying 
	\[\{u+\delta \le tu_{K_t}\}\subset B_{\delta}\subset \Omega.\]
	
	By Theorem \ref{comparison}, we have
	\[
	\int_{\{u_j+\delta<tu_{K_t}\}}S_k[u_j]\,dx\geq \int_{\{u_j+\delta<tu_{K_t}\}}S_k[tu_{K_{t}}]\,dx.
	\]
	Note that $u_j\ge u$. Hence
	\begin{align*}
		\limsup_j	\int_{\{u_j+\delta<tu_{K_t}\}}S_k[u_j]\,dx
		&\le \limsup_j\int_{\{u+\delta\le tu_{K_t}\}}S_k[u_j]\,dx\\
		&\le  \limsup_j\int_{B_\delta}S_k[u_j]\,dx\\
		&\le \int_{B_\delta}S_k[u]\,dx\le \int_{\Omega}S_k[u]\,dx.
	\end{align*}
	Since $u_j$ decrease and $\{u+\delta<-t\}\subset \bigcup_{j=1}^{\infty}\{u_j+\delta<-t\}$,  we get
	\begin{align*}
		\limsup_j\int_{\{u_j+\delta<tu_{K_t}\}}S_k[tu_{K_{t}}]\,dx
		&\ge  \limsup_j\int_{\{u_j+\delta<-t\}}S_k[tu_{K_{t}}]\,dx\\
		&\ge \int_{\{u+\delta<-t\}}S_k[tu_{K_{t}}]\,dx.
	\end{align*}
	Letting $\delta\to 0$, then 
	\[\int_{\{u+\delta<-t\}}S_k[tu_{K_{t}}]\,dx \to \int_{\{u<-t\}}S_k[tu_{K_{t}}]\,dx.\]
	Therefore we obtain
	\[
	\mathcal M_k(u)=\int_{\Omega}S_k[u]\,dx\geq t^{k}\int_{K_t}S_k[u_{K_{t}}]\,dx=t^{k}\operatorname{Cap}_{k}(K_{t},\Omega),\quad \forall\,t>0,
	\]
	that is,
	\begin{equation}\label{iso-cap}
		\operatorname{Cap}_k(K_t,\Omega)\leq \left(\frac{\mathcal M_k(u)^{1/k}}{t}\right)^k.
	\end{equation}

The rest of the proof is divided into two cases:

(1) Let $1\leq k<n/2$. 
By Theorem \ref{iso}(i), for  $1\leq q \leq \frac{n(k+1)}{n-2k}$, it holds
\[
|K_t| \leq C\left(\frac{\mathcal M_k(u)^{1/k}}{t}\right)^{\frac{kq}{k+1}}.
\]
Then for any $1\leq p \leq \frac{kq}{k+1}$,
\begin{align*}
	\int_{\Omega}\left(\frac{-u}{\mathcal M_k(u)^{1/k}}\right)^p\,dx
	&=\int_{\Omega}\chi_{ \{\frac{-u}{\mathcal M_k(u)^{1/k}} \leq 1\}}\cdot \left(\frac{-u}{\mathcal M_k(u)^{1/k}}\right)^p+\chi_{\{\frac{-u}{\mathcal M_k(u)^{1/k}}> 1\}} \cdot \left(\frac{-u}{\mathcal M_k(u)^{1/k}}\right)^p\,dx\\
	&\leq |\Omega|+\int_{\mathcal M_k(u)^{1/k}}^{\infty}pt^{p-1}\frac{|K_t|}{\mathcal M_k(u)^{p/k}}\,dt\\
	&\leq |\Omega|+Cp\int_{\mathcal M_k(u)^{1/k}}^{\infty}\mathcal M_k(u)^{\frac{q}{k+1}-\frac{p}{k}} t^{p-1-\frac{kq}{k+1}}\, dt\\ 
	&=|\Omega|+Cp\left(\frac{kq}{k+1}-p\right)^{-1}\cdot\mathcal M_k(u)^{\frac{q}{k+1}-\frac{p}{k}} \cdot (\mathcal M_k(u)^{1/k})^{p-\frac{kq}{k+1}}\\
	&\leq C.
\end{align*}

(2) Let $k=\frac{n}{2}$. By  Theorem \ref{iso}(ii), 
\[
|K_{t}|\leq C(n)|\Omega|\cdot\exp
\left(-\alpha_0\frac{1}{\operatorname{Cap}_{k}(K_t,\Omega)^{\frac{\beta}{k+1}}}\right)\leq C(n)|\Omega| \cdot\exp\left(-\alpha_0\frac{t^{\frac{k\beta}{k+1}}}{\mathcal M_k(u)^{\frac{\beta}{k+1}}}\right).
\]
Then for every $0<\lambda<\alpha_0$,
\begin{align*}
\int_{\Omega}\exp\left(\lambda\frac{(-u)^{\frac{k\beta}{k+1}}}{\mathcal M_k(u)^{\frac{\beta}{k+1}}}\right)\,dx
&=\sum_{j=0}^{\infty}\frac{\lambda^{j}}{j!}\int_{\Omega}\frac{(-u)^{\frac{k\beta  j}{k+1}}}{\mathcal M_k(u)^{\frac{\beta j}{k+1}}}\,dx\\
&=\sum_{j=0}^{\infty}\frac{\lambda^{j}}{j!}\int_{0}^{\infty}\frac{|K_{t}|}{\mathcal M_k(u)^{\frac{\beta j}{k+1}}}\,d(t^{\frac{k\beta j}{k+1}}) \\
&\leq C(n)|\Omega|\sum_{j=0}^{\infty}\frac{\lambda^{j}}{j!}\int_{0}^{\infty}\exp\left(-\alpha_0\frac{t^{\frac{k\beta}{k+1}}}{\mathcal M_k(u)^{\frac{\beta}{k+1}}}\right)\,d\left(\frac{t^{\frac{k\beta j}{k+1}}}{\mathcal M_k(u)^{\frac{\beta j}{k+1}}}\right) \\
&\leq C(n)|\Omega|\sum_{j=0}^{\infty}\frac{\lambda^{j}}{\alpha_0^{j}}=\frac{\alpha_0C(n)}{\alpha_0-\lambda}|\Omega|.
\end{align*}
\end{proof}

The following local Brezis-Merle type inequality will be used later in Section \ref{con-com}.
\begin{cor}\label{loc version of BM inequality}
	Suppose $u\in \mathcal F^k(\Omega)$ with $k=\frac{n}{2}$. Let $x_0\in \Omega$ satisfying $\mu_k[u](\{x_0\})<1$, then there is some $c>0$ such that $e^{-(\alpha_0+c) u} $ locally integrable near $x_0$.
\end{cor}

The proof needs a monotonicity formula for the non-commutative
\[\langle u, v\rangle:=\int_{\Omega} (-v) S_k[u].\]
Since we didn't find it in the literature, we present a proof here.

\begin{lemma}\label{hessian comparison}
	Let $u, v \in \mathcal{F}^k(\Omega)$ and $u \ge v$. Then for $h \in \mathcal{F}^k(\Omega)\cap L^\infty(\Omega)$, it holds
	$$\int_{\Omega} (-h) S_k[u] \le \int_{\Omega} (-h) S_k[v].$$
\end{lemma}
\begin{proof}
	Firstly we assume $u, v, h  \in \Phi_{0}^k(\Omega)$. 
	Let  
	\[S_k[u_1,..., u_k]=\frac{1}{k!}\sum_{ i_1\neq...\neq i_k}\lambda_{i_1}^1\cdots\lambda_{i_k}^k\]
	be the mixed Hessian, where $\lambda^i_1,...,\lambda^i_n$ are the eigenvalues of $D^2 u_i$. Using integration by parts, we have
	\begin{eqnarray*}
		\int_{\Omega} (-h) S_k[u]&=&\int_{\Omega} (-h) S_k[u, u, \dots, u]
		=\int_{\Omega} (-u) S_k[h, u, \dots, u]\\
		&\leq& \int_{\Omega} (-v) S_k[h, u, \dots, u]
		=\int_{\Omega} (-u) S_k[h, v, u , \dots, u]\\
		&\leq& \int_{\Omega} (-v) S_k[h, v, u, \dots, u]\\
		& \le& \dots \le \int_{\Omega} (-v) S_k[h, v, \dots, v] = \int_{\Omega} (-h) S_k[v].
	\end{eqnarray*}
	When $u,v, h$ are general $k$-convex functions, the monotonicity can be obtained by approximation with Lemma 2.5 in \cite{TW02}.
\end{proof}

\begin{proof}[Proof of Corollary \ref{loc version of BM inequality}]
	Let $B_j=B_{\frac{1}{j}}(x_0)$ and
	\[
	u_j := \sup \{ v \in \mathcal{F}^k(\Omega) : v \leq u \text{ on } B_j \}.
	\]
	Then $u_j \in \mathcal{F}^k(\Omega)$, $u_j \geq u$ , $u_j = u$ on $B_j$ and $\operatorname{supp}S_k [u_j]\subset B_{j-1}$. 
	
	Denote by $G(x,x_0)$ the Green function for $\Omega$ with logarithmic pole at $x_0$ and
	choose $\delta > 0$ small enough such that
	\[
	\int_{\Omega} (-\max\{\delta G(x,x_0), -1\}) S_k [u]< 1.
	\]
	By Lemma \ref{hessian comparison}, we have
	\[
	\int_{\Omega} (-\max\{\delta G(x,x_0), -1\}) S_k [u_j]
	\leq \int_{\Omega} (-\max\{\delta G(x,x_0), -1\}) S_k [u]< 1.
	\]
	If we choose $j$ such that $B_{j-1}\subset \{x: \delta G(x,x_0) < -1\}$, it follows that
	\[
	\int_{\Omega} S_k[u_j] = \int_{B_{j-1}} S_k[u_j]< 1.
	\]
	By the Brezis-Merle inequality, $\forall \alpha<\alpha_0$,
	\begin{eqnarray*}
		\int_{\Omega} \exp\left(\alpha\left(\frac{-u_j}{\mathcal M_k(u_j)^{\frac{1}{k}}}\right)\right)\,dx\leq C.
	\end{eqnarray*}
	Since $\mathcal M_k(u_j)<1$, the conclusion holds by choosing $\alpha$ sufficiently close to $\alpha_0$.
\end{proof}

\vskip 20pt

\section{An ABP type estimate}\label{ABP-pf}
The proof is inspired by a recent iterative method using an auxiliary Monge--Amp\`ere equation in the study of the complex
Monge--Amp\`ere equation \cites{GFT, Liu24}.

\begin{proof}[Proof of Theorem \ref{thm L infty estimate}]
Without loss of generality, we can assume $\sup_{\partial \Omega} u = 0$. 
Let $\{e^{G_i}\}$ be a sequence of smooth positive functions approximating $\chi_{\Gamma^-} S_k[-u]$ from above. Define
\[
N_i := \int_{\Omega} e^{G_i} \Phi(G_i) \,dx.
\]
Then $N_i \to \int_{\Gamma^-} S_k[-u] \Phi(\log(S_k[-u])) \,dx$ as $i \to \infty$.

Consider the auxiliary $k$-Hessian equation for $\psi_1$,
\[
\begin{cases}
S_k[\psi_1] = \frac{e^{G_i} \Phi(G_i)}{N_i} & \text{in } \Omega, \\
\psi_1 = 0 & \text{on } \partial \Omega.
\end{cases}
\]
By Theorem \ref{thm BM inequality}, there exists a constant $\alpha(n) > 0$  such that
\begin{align}\label{upper bd of exp-psi1}
\int_{\Omega} \exp(-\alpha(n) \psi_1) \,dx \le C(n)|\Omega|.
\end{align}

Let $q > 1$ and let $\alpha > 0$ be the constant from the inequality above. Define $h: \mathbb{R} \to \mathbb{R}$ as
\[h(s) = -\int_s^\infty \frac{q}{\alpha} \frac{N_i^{1/k}}{\Phi^{1/k}(t)} \,dt.\]
Since $\Phi$ is increasing, $h$ is a concave, increasing function.
Define the function $\psi = -h(-\frac{\alpha}{q} \psi_1)$. 
A direct computation shows
\[D^2\psi = \frac{\alpha}{q} h'(-\frac{\alpha}{q}\psi_1) D^2\psi_1-\frac{\alpha^2}{q^2}h'' D\psi_1\otimes D\psi_1.\]
We can get
\[
	S_k[\psi] \ge \left(\frac{\alpha}{q} h'(-\frac{\alpha}{q}\psi_1)\right)^k S_k[\psi_1] 
	= \left(\frac{\alpha}{q}\right)^k (h')^k \frac{e^{G_i} \Phi(G_i)}{N_i}.
\]
By the definition of $h$, we have $(h'(s))^k= \frac{q^k}{\alpha^k} \frac{N_i}{\Phi(s)}$. So
\[S_k[\psi] \ge  \frac{e^{G_i} \Phi(G_i)}{\Phi(-\frac{\alpha}{q}\psi_1)}.\]
If $G_i > -\frac{\alpha}{q}\psi_1$ , then $\Phi(G_i)>\Phi(-\frac{\alpha}{q}\psi_1)$, thus $e^{G_i}  \le S_k[\psi]$. If $G_i \le -\frac{\alpha}{q}\psi_1$ , $e^{G_i} \le \exp(-\frac{\alpha}{q}\psi_1)$.
In summary, we get
\begin{align}\label{G_k inequality}
e^{G_i} \,dx \le S_k[\psi] + F \,dx
\quad \text{with} \quad
F = \min\{\exp\{-\frac{\alpha}{q}\psi_1\}, e^{G_i}\}.
\end{align}

We consider the auxiliary $k$-Hessian equation
\[
\begin{cases}
S_k[\psi_{s,j}] = \frac{\eta_j(u+\psi-s) F}{A_{s,j}} & \text{in } \Omega, \\
	\psi_{s,j} = 0 & \text{on } \partial \Omega,
\end{cases}
\]
where 
\[A_{s,j} = \int_{\Omega} \eta_j(u+\psi-s) F \,dx\] 
and $\eta_j $ is a sequence of smooth positive functions approximating $\max\{x,0\}$ from above uniformly.

Consider the function $\Psi = u + \psi - s - \epsilon (-\psi_{s,j})^{\frac{k}{k+1}}$, where 
\[\epsilon=\left(\frac{k+1}{k}\right)^{\frac{k}{k+1}}A_{s,j}^{\frac{1}{k+1}}.\] 
Let $s_0=-h(0)\leq q\alpha^{-1}N_i^{1/k}\Lambda$. We claim that $\Psi\le 0$ for $s\ge s_0$. On the boundary, we have $\Psi\le 0$ when $s\ge -h(0)$. Without loss of generality, we assume that the maximum of $\Psi$ is attained at $x_0 \in \Omega_s = \{ x \in B_1 \mid u(x) + \psi(x) - s > 0 \}$. At $x_0$, we have
\begin{align*}
0&\ge D^2u+D^2\psi+\epsilon\frac{k}{k+1}(-\psi_{s,j})^{\frac{-1}{k+1}}D^2\psi_{s,j}.
\end{align*}
Then we get  
\[S_k[-u]\geq S_k[\psi]+\left(\epsilon\frac{k}{k+1}\right)^k(-\psi_{s,j})^{\frac{-k}{k+1}}_iS_k[\psi_{s,j}].\]
Using the inequality (\ref{G_k inequality}), we have
\begin{align*}
S_k[\psi] +F&\geq e^{G_i}\ge S_k[-u]\\
&\geq  S_k[\psi]+\left(\epsilon\frac{k}{k+1}\right)^k(-\psi_{s,j})^{\frac{-k}{k+1}}S_k[\psi_{s,j}].
\end{align*}
By the definition of $\psi_{s,j}$, we get
\[\eta_j(u+\psi-s)\le A_{s,j}\left(\epsilon\frac{k}{k+1}\right)^{-k}(-\psi_{s,j})^{\frac{k}{k+1}}.\]
From the definition of $\epsilon$, it follows $\Psi\leq 0$.

Let $\phi(s) := \int_{\Omega_s} F \,dx$, where $\Omega_s = \{u + \psi > s\}$ and $F = \min(e^{-\frac{\alpha}{q}\psi_1}, e^{G_i})$.
Let 
\[A_s := \int_{\Omega_s} (u + \psi - s) F \,dx.\]
For $t > 0$, we have $(u+\psi-s) \ge t$ on $\Omega_{s+t}$ . Thus,
\[
t \cdot \phi(s+t) = t \int_{\Omega_{s+t}} F \,dx \le \int_{\Omega_{s+t}} (u+\psi-s) F \,dx \le \int_{\Omega_s} (u+\psi-s) F \,dx = A_s.
\]
For $s\ge s_0, \Psi \le 0$, we have $u+\psi-s \le \epsilon(-\psi_{s,j})^{\frac{k}{k+1}}$. Then 
\begin{eqnarray*}
A_s &=& \int_{\Omega_s} (u+\psi-s) F \,dx \\
&\le& \epsilon \int_{\Omega_s} (-\psi_{s,j})^{\frac{k}{k+1}} F \,dx\\
&\le& \epsilon \left( \int_{\Omega_s} (-\psi_{s,j})^{\frac{kp}{k+1}} F \,dx \right)^{1/p} \left( \int_{\Omega_s} F \,dx \right)^{1-1/p}.
\end{eqnarray*}
Letting $j\to \infty$, we have $A_{s,j}\to A_s$. Note that $\epsilon=c(n)A_{s,j}^{\frac{1}{k+1}}\to c(n)A_{s}^{\frac{1}{k+1}}$, then we get
\[
A_s 
\le c(n) \left( \int_{\Omega_s} (-\psi_{s,j})^{\frac{kp}{k+1}} F \,dx \right)^{\frac{k+1}{kp}} \left( \int_{\Omega_s} F \,dx \right)^{(1-1/p)\frac{k+1}{k}}.
\]
For $p>k+1$, let $\delta=(1-1/p)\frac{k+1}{k}-1>0$.
By \eqref{upper bd of exp-psi1}, we have
\[\int_{\Omega_s} (-\psi_{s,j})^{\frac{kp}{k+1}} F \,dx\leq \left(\int_{\Omega}(-\psi_{s,j})^{\frac{kpq^{*}}{k+1}}\,dx\right)^{1/q^{*}}||F||_{L^q}\le c(n, |\Omega|).\]
Here $q^*$ sasisfies $1/q^{*}+1/q=1$.
Hence
\[A_s \le C(n) \cdot \phi(s)^{1+\delta}.\]
This gives 
\[t \phi(s+t) \le C \phi(s)^{1+\delta}, \quad \forall s\ge s_0\; \text{and} \; t>0.\]
The following is a classic lemma due to De Giorgi \cite{De57}, which was also used in \cites{Ko98, EGZ, GFT}.
\begin{lemma}[De Giorgi's lemma]
	Let $\phi : \mathbb{R} \to \mathbb{R}$ be a monotone decreasing function such that for some $\delta > 0$ and any $s \ge s_0$, $t > 0$,
	\[
	t\,\phi(s+t) \le C_0\,\phi(s)^{1+\delta}.
	\]
	Then $\phi(s) = 0$ for any $s \ge \dfrac{2C_0\,\phi(s_0)^{\delta}}{1 - 2^{-\delta}} + s_0$.
\end{lemma}
Applying De Giorgi's lemma, we obtain $\phi(s) = 0$ for all $s \ge s_\infty$, where
\[
s_\infty = \frac{2 C \phi(s_0)^\delta}{1-2^{-\delta}} + s_0.
\]
This means $u+\psi \le s_\infty$, then
\[\sup_{\Omega} u \le \sup_{\Omega} (u+\psi) \le s_\infty.\]
Note that
\[s_0 = -h(0) = \int_0^\infty \frac{q}{\alpha} \frac{N_i^{1/k}}{\Phi^{1/k}(t)} \,dt = c_2(n,\Lambda) N_i^{1/k}\]
and
\[\phi(s_0) \le \int_{\Omega} F \,dx \le \int_{\Omega} e^{-\frac{\alpha}{q}\psi_1} \,dx \le c_1(n, |\Omega|).\]
Hence, we obtain 
\[s_\infty \le c_1 + c_2 N_i^{1/k}.\]
Letting $i \to \infty$, we have
\[
\sup_{\Omega} u \le c_1+ c_2 \left( \int_{\Gamma^-} S_k[-u] \Phi(\log(S_k[-u])) \,dx \right)^{1/k}.
\]
This completes the proof.
\end{proof}

\vskip 20pt

\section{The mean field type equation and functionals}\label{mean-field eq}

In the following two sections we apply the inequalities and estimates in former sections to 
$k$-Hessian equation with $k=\frac{n}{2}$. We first introduce the associated functionals for the mean field equation and prove some basic properties. The framework is similar to the complex Monge-Amp\`ere equation \cite{BB}.

\begin{lemma}\label{increase of G}
$G_{\gamma_1}(u)\le G_{\gamma_2}(u)$ for $0 <\gamma_1 < \gamma_2$.
\end{lemma}
\begin{proof}
 Let $\theta = \frac{\gamma_1}{\gamma_2} \in (0,1)$. 
	By H\"older's inequality  we obtain
	\[
	\int_\Omega e^{-\gamma_1 u}\, dx
	\le \left( \int_\Omega e^{-\gamma_2 u}\, dx \right)^\theta \left( \int_\Omega 1\, dx \right)^{1-\theta}.
	\]
	Dividing both sides by $|\Omega|$ yields
	\[
	\frac{1}{|\Omega|} \int_\Omega e^{-\gamma_1 u}\, dx  \le \left( \frac{1}{|\Omega|} \int_\Omega e^{-\gamma_2 u}\,dx  \right)^\theta.
	\]
	Taking the logarithm we obtain
	\[
	\frac{1}{\gamma_1} \log\left( \frac{1}{|\Omega|} \int_\Omega e^{-\gamma_1 u}\,dx  \right) \le \frac{1}{\gamma_2} \log\left( \frac{1}{|\Omega|} \int_\Omega e^{-\gamma_2 u}\,dx  \right).
	\]
Thus we get $G_{\gamma_1}(u)\le G_{\gamma_2}(u)$. 
\end{proof}

Given a non-negative measure $\mu$, following \cite{BB}, we define
\begin{align*}
	D(\mu):=\int_{\Omega} \log (\mu/dx) \,\mu
\end{align*} 
if $\mu$ is absolutely continuous with respect to the Lebesgue measure~$dx$ (with density $\mu/dx$) and otherwise $D(\mu)=\infty$. We also set
\begin{align*}
\mathcal{L}_{\gamma}(u):=-\frac{1}{\gamma}\log \int_{\Omega}e^{-\gamma u}dx.
\end{align*}
According to \cite{Ber13}, the quantity $D(\mu)/\gamma$ can be realized as a Legendre-type transform of $\mathcal{L}_{\gamma}$, which yields the dual variational formulas
\begin{align}
\frac{	D(\mu)}{\gamma}&=\sup_{u\in C^0(\bar{\Omega})}\left\{-\frac{1}{\gamma}\log \int_{\Omega}e^{-\gamma u}-\int_\Omega u\,\mu\right\}, \label{D functional}\\[5pt]
\mathcal{L}_{\gamma}(u)&=\inf_{\mu} \left\{\frac{	D(\mu)}{\gamma}+\int_\Omega u\,\mu\right\}. \label{L functional}
\end{align}
The supremum in \eqref{D functional} and the infimum in \eqref{L functional} are attained when
$\displaystyle \mu = \frac{e^{-\gamma u}}{\int_{\Omega} e^{-\gamma u}\,dx}\,dx$.

Assume there is $u_{\mu}\in \Phi_0^k(\Omega)$ such that $S_k[u_{\mu}]\,dx=\mu$. Then we can define the free energy functional 
\[F_{\gamma}(\mu):=-\frac{k}{k+1}\int_\Omega u_{\mu}\,\mu-\frac{D(\mu)}{\gamma}.\]
If $S_k[u]=\frac{e^{-\gamma u}}{\int_{\Omega} e^{-\gamma u}\,dx}\,dx$ , we have
\begin{align}\label{F G eq}
G_{\gamma}(u)= F_{\gamma}\left(\frac{e^{-\gamma u}\,dx}{\int_\Omega e^{-\gamma u}\,dx}\right)-\frac{1}{\gamma}\log |\Omega|.
\end{align}

\vskip 20pt

\section{Concentration-compactness principle}\label{con-com}

In this section, we prove the concentration-compactness principle for the 
$k$-Hessian equation of mean field type. 

We need the Brezis-Merle type inequality as well as the following Moser--Trudinger type inequality for the 
$k$-Hessian equations.

\begin{theorem}[\cite{Wan94, Wan09}]\label{thm MT type inequality} Let $u \in \Phi_0^k(\Omega)$.
	\begin{enumerate}
		\item [(1)]  If $1\leq k<\frac{n}{2}$, $\forall p+1\in [1,k^*]$, we have
		\[\|u\|_{L^{p+1}(\Omega)}\leq C\|u\|_{\Phi_0^k(\Omega)},\]
		where $C$ depends only on $n$, $k$, $p$, and $|\Omega|$, $k^*=\frac{n(k+1)}{n-2k}$.
		
		\item [(2)] Let $k = \frac{n}{2}, 0<\alpha\le \alpha_0, 1\le \beta\le \beta_0$. Then
		\begin{align}\label{MT type inequality}
			\sup \left\{ \int_{\Omega} \exp \left( \alpha \left( \frac{-u}{\|u\|_{\Phi_0^k(\Omega)}} \right)^\beta \right) \,dx : u \in \Phi_0^k(\Omega) \right\} < C,
		\end{align}
		where $\alpha_0$, $\beta_0$ are given by \eqref{apbe} and $C$ is a positive constant depending only on $n$ and $\operatorname{diam}(\Omega)$.
	\end{enumerate}
	
\end{theorem}

\begin{proof}[Proof of Theorem~\ref{mean-field-compactness}]
We distinguish two cases.  
\medskip

\noindent\textbf{Case 1:} There exists $\delta>0$ such that
\begin{align}\label{regular condition}
	\int_{\Omega}e^{-(\alpha_0+\delta)u_j}\,dx\leq C_\delta\quad
	\text{for all } j.
\end{align}
In this case we will prove that alternative (1) of the theorem holds.
	
Denote $\mu_j=S_k[u_j]\,dx$. 
For $t<\alpha_0+\delta$, we have $\mathcal{L}_{t}(u_j)\ge -C_t$. From (\ref{D functional}),
\begin{align*}
	\frac{D(\mu_j)}{t}\ge \mathcal{L}_{t}(u_{j})-\int_\Omega u_{j}\,\mu_j\,\ge -C_t-\int_\Omega u_{j}\,\mu_j.
\end{align*}
Thus
\begin{align*}
F_{\gamma}(\mu_j)
&=-\frac{k}{k+1}\int_\Omega u_{j}\,\mu_j-\frac{D(\mu_j)}{\gamma}\\
&\le \left(\frac{t}{\gamma}-\frac{k}{k+1}\right)\int_\Omega u_{j}\,\mu_j + \frac{t}{\gamma}C_t.
\end{align*}
For any $\gamma<\frac{k+1}{k}(\alpha_0+\delta)$, we can find $0<t<\alpha_0+\delta$ such that $\frac{t}{\gamma}-\frac{k}{k+1}>0$. Therefore there exists $\delta'>0$ such that 
\begin{align}\label{F critical upper bd}
	F_{\frac{k+1}{k}\alpha_0+\delta'}(\mu_j)\le C.
\end{align}

By Lemma \ref{increase of G} and \eqref{F G eq}, we get 
\begin{align}
		-C:=G_{\gamma_1}(u_1)
		&\le G_{\gamma_j}(u_j)=F_{\gamma_j}(\mu_j)+\frac{\log |\Omega|}{\gamma_j}.\nonumber
\end{align}
Thus
\begin{align}\label{Fj lower bd}
	F_{\gamma_j}(\mu_j)\ge -C' .
\end{align}
	By definition,
	\begin{eqnarray}
	D(\mu_j)&=&-\gamma_jF_{\gamma_j}(\mu_j)-\frac{k\gamma_j}{k+1}\int_\Omega u_jS_k[u_j],\label{D expression}
	\end{eqnarray}
	Using \eqref{D expression}, we have
	\begin{align}
	F_{\frac{k+1}{k}\alpha_0+\delta'}(\mu_j)
	&=-\frac{k}{k+1}\int_\Omega u_jS_k[u_j]-\frac{D(\mu_j)}{\frac{k+1}{k}\alpha_0+\delta'}\nonumber\\
	&=\left(\frac{\gamma_j}{\frac{k+1}{k}\alpha_0+\delta'}-1\right)\frac{k}{k+1}\int_\Omega u_jS_k[u_j]+C_1F_{\gamma_j}(u_j).\label{F energy rela}
	\end{align}
	Note that \[\frac{\gamma_j}{\frac{k+1}{k}\alpha_0+\delta'}-1\le -C_{\delta'}<0.\] 
	Combining \eqref{F critical upper bd} , \eqref{Fj lower bd} and \eqref{F energy rela}, we obtain $||u_j||_{\Phi_{0}^k(\Omega)}$ is uniformly bounded.
	
	By Moser-Trudinger inequality, we get $\int_\Omega e^{-pu_j}\,dx\le C_p$ for any $p>0$. 
	Since $e^{-u_j}\to e^{-u}$ in measure and $e^{-pu_j}$ is uniformly integrable, the Vitali convergence theorem implies that $\int_\Omega e^{-p u_j}\,dx\to \int_\Omega e^{-p u}\,dx$.
	Thus $u$ is a weak solution of equation \eqref{mean field equation of k}  for $\gamma=\frac{k+1}{k}\alpha_0$.
	
\medskip
\noindent\textbf{Case 2:}	If \eqref{regular condition} does not hold, i.e., for any
$\delta > 0$,  the sequence $\int_{\Omega}e^{-(\alpha_0+\delta)u_j}\,dx$ is unbounded. Note that $u_j\to u$ in $L^1_{\text{loc}}(\Omega)$. We choose a subsequence $u_{j_k}\to u$ a.e. in $\Omega$. Since there is no blow-up point on the boundary, there exists a compact $K \subset \Omega$ and $M>0$ such that $u>-M$ a.e on $\Omega\setminus K$.  Since $\int_{\Omega}e^{-(\alpha_0+\delta)u_{j_k}}\,dx$ is unbounded, we have	
	\begin{align}\label{singular condition}
		\int_{K}e^{-(\alpha_0+\delta)u}\,dx=+\infty.
	\end{align}
	
	Now we claim there is $x_0\in K$ such that $\mu_k[u](\{x_0\})\ge 1$.
	Otherwise for any $x\in K$, $\mu_k[u](\{x\})< 1$. By Corollary \ref{loc version of BM inequality}, there is $r_x, c_x>0$ such that 
	\[\int_{B_{r_x}(x)}e^{-(\alpha_0+c_x)}\,dx<+\infty.\] 
	Since $K\subset \cup_{x\in K}B_{r_x}(x)$, so we can find finite balls such that $K\subset \cup_{ i=1}^mB_{r_{x_i}}(x_i)$. Thus we get some $\delta'>0$ and 
	\[\int_{K}e^{-(\alpha_0+\delta')u}\,dx<+\infty,\] 
	which contradicts \eqref{singular condition}. 
	
	Note that $u_j\to u$ in $L^1_{\text{loc}}(\Omega)$ and $\int_{\Omega}S_k[u_j]\,dx=1$, so we have 
	$\mathcal M_k(u)\leq 1$. Therefore we have $\mu_k[u](\{x_0\})=1$, i.e., $\mu_k[u]=\delta_{x_0}$. Then by Theorem 3.6 in \cite{Lab02},
	\[u(x)=\left[\tbinom{n}{k}\omega_n\right]^{\frac{1}{k}}\log|x-x_0|+O(1).\]
\end{proof}

\end{document}